\documentclass[12pt,a4paper,leqno,verbatim]{amsart}

\newcounter{minutes}
\divide\time by 60
\newcounter{hours}
\multiply\time by 60 \addtocounter{minutes}{-\time}

\usepackage{amssymb}
\usepackage{hyperref}
\usepackage{graphicx}
\date{}
\newfont{\cyrilic}{wncyr10 scaled 1000}
\title{On functional inequalities for the psi function}
\author[B. A. Bhayo]{Barkat Ali Bhayo}
\address{Department of Mathematics \& Statistics,
Quaid-e-Awam University of Engineering Science \& Technology Nawabshah, Pakistan}
\email{bhayo.barkat@gmail.com}

\newcommand{\comment}[1]{}

\swapnumbers
\theoremstyle{plain}

\newtheorem{theorem}[equation]{Theorem}
\newtheorem{lemma}[equation]{Lemma}

\newtheorem{corollary}[equation]{Corollary}

\newtheorem{subsec}[equation]{}

\numberwithin{equation}{section}

\begin{document}
\font\fFt=eusm10 
\font\fFa=eusm7  
\font\fFp=eusm5  
\def\K{\mathchoice
{\hbox{\,\fFt K}}
{\hbox{\,\fFt K}}
{\hbox{\,\fFa K}}
{\hbox{\,\fFp K}}}

\def\thefootnote{}
\footnotetext{ \texttt{\tiny File:~\jobname .tex,
          printed: \number\year-\number\month-\number\day,
}} \makeatletter\def\thefootnote{\@arabic\c@footnote}\makeatother
\maketitle

\begin{abstract} In this note we study the monotonicity of the function $x\mapsto\psi(1+bx)^a/\psi(1+ax)^b$. We also
give the several inequalities involving the psi function, whic is the logarithmic derivative of the gamma function.
\end{abstract}

\bigskip

\section{Introduction}
For ${\rm Re}\, x>0$, we define the classical \emph{gamma function} $\Gamma(x)$ and the $\emph{psi function}$ $\psi(x)$ by
$$\Gamma(x)=\int^\infty_0 e^{-t}t^{x-1}\,dt,\,\,\psi(x)=\frac{\Gamma^{'}(x)}{\Gamma(x)},$$
respectively. The recurrence relations of $\Gamma$ and $\psi$ are
$$\Gamma(1+x)=x\Gamma(x),\quad \psi(x+1) = \frac{1}{x} + \psi(x).$$
Note that
$$\psi(1)=-\gamma\quad {\rm and}\quad \psi(1/2)=-2\log 2-\gamma$$
where $\gamma$ is the $\emph{Euler-Mascheroni}$ constant.
Throughout this paper, we denote by $c = 1.461632144968362\ldots$
the only positive root of the the equation  $\psi(x) = 0 $ (see \cite[6.3.19]{as}).

In 2006, C. Alsina and M.S. Tomas \cite{at} prove the following interesting inequality involving the gamma function

\begin{equation}\label{aliseq}
\frac{1}{n!}\leq \frac{\Gamma(1+x)^n}{\Gamma(1+nx)},\quad x\in[0,1],\,n=1,2,\ldots,
\end{equation}
by using a geometrical method. Motivated by their result, J. S\'andor \cite{s} extended the inequality \eqref{aliseq}
as follows:
$$\frac{1}{\Gamma(1+a)}\leq \frac{\Gamma(1+x)^a}{\Gamma(1+ax)},\quad x\in[0,1],\,a\geq 1.$$
In \cite{me}, Mercer obtained the following inequalities
$$\frac{\Gamma(1+x)^a}{\Gamma(1+ax)}< \frac{\Gamma(1+y)^a}{\Gamma(1+ay)},\quad a\in(0,1),$$
$$\frac{\Gamma(1+x)^a}{\Gamma(1+ax)}> \frac{\Gamma(1+y)^a}{\Gamma(1+ay)},\quad a\in\mathbb{R}\setminus(0,1),$$
with $0<x<y,\,1+ax>0$, and $1+ay>0$. In 2006, Bougotta \cite{bo} proved the monotonicity property of
$x\mapsto\Gamma(1+bx)^a/\Gamma(1+ax)^b$, by using the same method S\'andor \cite{s}. It was pointed out by the referee of this paper pointed that the inequalities of S\'andor \cite{s} and Mercer \cite{me} follow from the inequality (2.4)
in \cite{ne2}. For the related inequalities of the gamma function, we refer the reader to see \cite{nn}.


Our first result is the counterpart of the above results and reads as follows.
\begin{theorem}\label{r1} For $a,b>1$, the following function
$$f(x)=\frac{\psi(1+bx)^a}{\psi(1+ax)^b}$$
is increasing for $a>b$ and decreasing for $a<b$ in $(c-1,\infty)$, respectively.

In particular, for $1<b<a$
$$\displaystyle\left(\frac{\psi(1+bx)}{\psi(1+b(c-1))}\right)^a >\left(\frac{\psi(1+ax)}{\psi(1+a(c-1))}\right)^b,$$
and the reverse inequality holds for $1<a<b$.
\end{theorem}

\begin{theorem}\label{r4} The function $g(x)=1/\psi(\cosh(x))$ is decreasing and convex from $(c_1,\infty)$
onto $(1/\psi(\cosh(c_1)),0)$, where $c_1={\rm arccosh}(c)=0.92728\ldots$. In particular,
$$\frac{2\psi(r)\psi(s)}{\psi((\sqrt{(1+r)(1+s)}+\sqrt{(r-1)(s-1)})/2)}
\leq \psi(r)+\psi(s),$$
for all $r,s\in(c,\infty)$, equality holds for $r=s$.
\end{theorem}



\section{Preliminaries and proofs}

The following lemma will be used in our proofs, which can be found in \cite{al,el, alzer, egp, sz}.
\begin{lemma}\label{lem1} For $x>0$ we have\\

\begin{enumerate}

\item $\displaystyle\log x-\frac{1}{x}<\psi(x)<\log x-\frac{1}{2x}$,\\

\item $\displaystyle\psi'(x)>\frac{1}{x}+\frac{1}{2x^2}$,\\

\item $\displaystyle\psi''(x)<\frac{1}{x}-2\psi'(x)$,\\

\item $(\psi'(x))^2+\psi''(x)>0$,\\

\item $2\psi'(x)+x\psi''(x)<\displaystyle\frac{1}{x}$.
\end{enumerate}
\end{lemma}

\begin{subsec}{\bf Proof of Theorem \ref{r1}.} \rm Let $$g(x)=\log f(x)=a\log (\psi(1+bx))-b\log(\psi(1+ax)).$$
Differentiating $g$ with respect to $x$ we get
$$g'(x)=ab\left(\frac{\psi'(1+bx)}{\psi(1+bx)}-\frac{\psi'(1+ax)}{\psi(1+ax)}\right).$$
It is easy to see that the function $\psi'(z)/\psi(z)$ is positive and decreasing for $z\in(c,\infty)$.
This implies that $g'(x)$ for $x>c-1$ is positive (negative) when $1<b<a\,(1<a<b)$. The proof follows from this observation.
$\hfill\square$
\end{subsec}

\begin{lemma}\label{lem4b} The function
$$f(x)=\frac{\psi '(\cosh (x)) \sinh (x)}{\psi (\cosh (x))^2},$$
is decreasing in $x\in(c_1,\infty)$.
\end{lemma}

\begin{proof} Letting $r=\cosh(x)$,
$$f(x)=\frac{\sqrt{r^2-1}\,\psi'(r)}{(\psi(r))^2}.$$
Differentiating $f$ with respect $x$ we get by Lemma \ref{lem1} (4) and (5) 
\begin{eqnarray*}
f'(x)&=&\frac{-2(r^2-1) \psi '(r)^2+\psi (r) \left(r \psi '(r)+\left(r^2-1\right) \psi ''(r)\right)}{
   \psi (r)^3}\\
   &<&\frac{2(r^2-1) \psi ''(r)+\psi (r) \left(r \psi '(r)+\left(r^2-1\right) \psi ''(r)\right)}{
   \psi (r)^3}\\
    &=&\frac{(r^2-1) \psi ''(r)(2+\psi(r))+r\psi(r)\psi'(r)}{\psi (r)^3}\\
    &<&\frac{(r-1/r) (1/r-2\psi '(r))(2+\psi(r))+r\psi(r)\psi'(r)}{\psi (r)^3}=\frac{f_1(x)}{\psi (r)^3}.
\end{eqnarray*}
Clearly $\psi(r)^3$ is positive because $r>c$, in order to show that $f_1$ is negative or equivalently
$$(r-1/r) (1/r-2\psi '(r))(2+\psi(r))<-r\psi(r)\psi'(r),$$
it is enought to prove that
\begin{equation}\label{neq1}
\frac{1}{r}-2\psi'(r)<-\psi'(r),
\end{equation}
and
\begin{equation}\label{neq2}
(r-1/r)(2+\psi(r))>r\psi(r).
\end{equation}
The inequality \eqref{neq1} is valid and follows from Lemma \ref{lem1}(2). To prove inequality \eqref{neq2},
we get by Lemma \ref{lem1}(1)
$$\frac{r^2-1+(r^2-2)\psi(r)}{r}>\frac{r^2-1+(r^2-2)(\log r-1/r)}{r}=\frac{f_2(r)}{r}.$$
The function $f_2$ is positive, because
$$f_2'(r)=2r\log r+3r-1-2/r^2-2/r>0,$$
and $\lim_{r\to c}f_2(r)=2.18993\ldots$. Hence $f'$ is negative, and consequently $f$ is decreasing, this completes the proof.
\end{proof}

\begin{subsec}{\bf Proof of Theorem \ref{r4}.}
\rm Differentiating $g$ with respect to $x$ we get
$$g'(x)=-\frac{\psi '(\cosh (x)) \sinh (x)}{\psi (\cosh (x))^2},$$
which is negative and decreasing, hence $g$ is convex. This implies that
$$\frac{1}{2}\left(\frac{1}{\psi(\cosh(x))}+\frac{1}{\psi(\cosh(y))}\right)\leq \frac{1}{\psi(\cosh(x+y)/2)}$$
$$=1/\psi\left(\frac{\sqrt{(\cosh(x)+1)(\cosh(y)+1)}}{2}
+\frac{\sqrt{(\cosh(x)-1)(\cosh(y)-1)}}{2}\right).$$
Setting $r=\cosh(x)$ and $s=\cosh(y)$ we complete the proof.
$\hfill\square$
\end{subsec}

Before we formulate the following result, we denote $z'=\sqrt{1-z^2},\,z\in(0,1)$.
The functional inequality of the following corollary is reminiscent of Theorem 5.12 of \cite{avvb}.

\begin{corollary} The following inequality 
$$\psi(r)+\psi(s)\leq 2\psi\left(\sqrt{\frac{2rs}{1+rs+r's'}}\right)$$
holds for $r,s\in(0,1)$, with equality for $r=s$.
\end{corollary}

\begin{proof} Let $h=\psi(1/\cosh(x))$ for $x>0$. We get
$$h'(x)=-\psi'(1/\cosh(x))\frac{\tanh(x)}{\cosh(x)}=-\tanh(x)h_1(x).$$ Let $u=1/\cosh(x)$,
one has 
$$h_1'(x)=-uu'(\psi'(u)+r\psi''(u)),$$
which is positive by Lemma \ref{lem1}. Thus, $h_1$ is increasing, and also 
$\tanh(x)$ is increasing.  Clearly, $h'$ is decreasing and negative, hence $h$ is concave in $x>0$. This implies 
$$\psi(1/\cosh((x+y)/2))\geq \frac{\psi(1/\cosh(x))+\psi(1/\cosh(y))}{2}.$$
The desired inequality follows if we let $r=1/\cosh(x),s=1/\cosh(y)$ and use the identity $\cosh^2((x+y)/2)=(1+xy+x'y')/(2 xy)$.
\end{proof}

For convenience we use the notation $\mathbb{R}_+=(0,\infty)$.
\begin{lemma}\label{neu}\cite[Thm 2.1]{ne2} Let $f:\mathbb{R}_+\to \mathbb{R}_+$
be a differentiable, log-convex function and let $a\geq 1$. Then $g(x)=(f(x))^a/f(a\,x)$
 decreases on its domain. In particular, if $0\leq x\leq y\,,$ then the following inequalities
 $$\frac{(f(y))^a}{f(a\,y)}\leq\frac{(f(x))^a}{f(a\,x)}\leq (f(0))^{a-1}$$
 hold true. If $0<a\leq 1$, then the function $g$ is an increasing function on $\mathbb{R}_+$
and inequalities are reversed.
\end{lemma}

\begin{corollary}\label{r2} For $k>1$ and $c<x\leq y$, the following inequality holds
$$\left(\frac{\psi(x)}{\psi(y)}\right)^k\leq \frac{\psi(kx)}{\psi(ky)}.$$
\end{corollary}

\begin{proof}
Let $g_1(x)=\log(1/\psi(x))$. Differentiating $g_1$ twice with respect to $x$ we get by Lemma \ref{lem1} 
\begin{eqnarray*}
g''(x)&=&\frac{\psi'(x)^2-\psi(x)\psi(x)^2}{\psi''(x)}\\
&>&\frac{-\psi''(x)(1+\psi(x))}{\psi(x)^2}>0,\\
\end{eqnarray*}
which implies that $g_1$ is convex. Now the rest of proof follows easily from Lemma \ref{neu}.
\end{proof}

\begin{theorem}\label{r5} The function $f(x)={\rm artanh}(\psi(\tanh(x)))$ is strictly
 increasing and concave from $(c,\infty)$ onto $(l,m)$, where
 $$f(c)={\rm artanh}(\psi(\tanh(c)))=-0.9934\ldots=l$$ and
 $$\quad f(\infty)=-{\rm artanh}(\gamma)=-0.6582\ldots=m.$$
In particular,
\begin{enumerate}
\item $\displaystyle\psi\left(\frac{r+s}{1+rs+r's'}\right)> \frac{\psi(r)+\psi(s)}
{1+\psi(r)\psi(s)+\sqrt{1-\psi(r)^2}\sqrt{1-\psi(s)^2}}$,\\
for all $r,s\in(0,1)$,\\

\item $\displaystyle\psi\left(\frac{r+s}{1+rs}\right)> \tanh\left(\frac{1}{2}\,{\rm artanh}\left(\frac{\psi(2r)+\psi(2s)}{1+\psi(2r)\psi(2s)}\right)\right)$,\\
for all $r,s\in(0,1)$,\\

\item $\displaystyle\frac{1+\psi(\tanh(r))}{1-\psi(\tanh(r))}\frac{1-\psi(\tanh(s))}{1+\psi(\tanh(s))}<
e^{2a(r-s)},$\\
for all $r,s\in(c,\infty)$, where
$$f'(c)=\frac{\psi'(\tanh (c)) {\rm sech}^2(c)}{1-\psi(\tanh(c))^2}=0.8807\ldots=a.$$
\end{enumerate}
\end{theorem}

\begin{proof} Differentiating $f$ with respect to $x$, we get
$$f'(x)=\frac{\psi'(\tanh (x)) \text{sech}^2(x)}{1-\psi(\tanh(x))^2}=\frac{F(x)}{G(x)}.$$
We see that $f'$ is positive and decreasing, because
$$F'(x)=\psi ''(\tanh (x)) \text{sech}^4(x)-2 \psi'(\tanh (x))
   \text{sech}^2(x) \tanh (x)<0.$$
Clearly $G(x)$ is increasing, hence $f$ is concave.
The concavity of the function implies that
$$f\left(\frac{x+y}{2}\right)>\frac{f(x)+f(y)}{2}$$
$$\Longleftrightarrow \psi\left(\tanh\left(\frac{x+y}{2}\right)\right)>\tanh\left(\frac{{\rm artanh}(\psi(\tanh(x)))
+{\rm artanh}(\psi(\tanh(y)))}{2}\right).$$
We get (1) by using $$\tanh\left(\frac{x+y}{2}\right)=\frac{\tanh(x+y)}{1+\sqrt{1-\tanh^2(x+y)}}=\frac{r+s}{1+rs+r's'},$$
$$\tanh\left(\frac{u+v}{2}\right)=\frac{R+S}{1+RS+R'S'},$$
and letting $r=\tanh(x),\,s=\tanh(y),\,R=\psi(r),\,S=\psi(s),
\,u={\rm artanh}(\psi(\tanh(x)))$ and $v={\rm artanh}(\psi(\tanh(y)))$.
Let us next consider (2), we obtain
\begin{eqnarray*}
{\rm artanh}\psi\left(\tanh\left(\frac{x+y}{2}\right)\right)&>&\frac{{\rm artanh}(\psi(\tanh(x)))
+{\rm artanh}(\psi(\tanh(y)))}{2}\\
&=&\frac{1}{2}{\rm artanh}\left(\frac{\psi(\tanh(x))+\psi(\tanh(y))}{1+\psi(\tanh(x))\psi(\tanh(y))}\right).
\end{eqnarray*}
Letting $r=\tanh(x/2)$ and $s=\tanh(y/2)$ we get (2).
The derivative $f'(x)$ tends to $a$ when $x$ tends to $c$. By Mean Value Theorem we get
$f(r)-f(s)<a(r-s).$
This is equivalent to
$$\frac{1}{2}\log\left(\frac{1+\psi(\tanh(r))}{1-\psi(\tanh(r))}\right)-\frac{1}{2}
\log\left(\frac{1+\psi(\tanh(s))}{1-\psi(\tanh(s))}\right)<a
(r-s),$$ hence (3) follows, and this completes the proof.
\end{proof}

\begin{lemma}\label{rv}\cite[Thm 1.7]{kmsv} Let $f:\mathbb{R}_+\to \mathbb{R}_+$ be a differentiable function
and for $c_2\neq 0$ define
$$g(x)=\frac{f(x^{c_2})}{(f(x))^{c_2}}\,.$$
We have the following
\begin{enumerate}
\item if $h(x)=\log(f(e^x))$ is a convex function, then $g(x)$ is monotone increasing for $c_2,x\in(0,1)$ or $c_2,x\in(1,\infty)$ or
$c<0,x>1$and monotone decreasing for $c_2\in(0,1), x>1$  or $c_2>1,\,x\in(0,1)$ or $c_2<0,x\in(0,1)$,\\
\item if $h(x)$ is a concave function, then $g(x)$ is monotone increasing for
 $c_2\in(0,1),x>1$ or $c_2>1,\,x\in(0,1)$ or $c_2<0,\,x\in(0,1)$
and monotone decreasing for $c_2,x\in(0,1)$ or $c_2>1,x>1$ or $c_2<0,\,x>1$.
\end{enumerate}
\end{lemma}

\begin{lemma}\label{lem0}\cite[Lemma 2.1]{bari} Let us consider the function $f : (a,\infty)\in\mathbb{R}$, where $a\geq 0$. If the
function $g$, defined by
$$g(x) =\frac{f(x) - 1}{x}$$
is increasing on $(a,\infty)$, then for the function $h$, defined by $h(x) = f(x^2)$, we
have the following Gr\"unbaum-type inequality
\begin{equation}\label{lemm01}
1 + h(z) \geq  h(x) + h(y),
\end{equation}
where $x, y \geq a$ and $z^2 = x^2+y^2$. If the function $g$ is decreasing, then inequality
(\ref{lemm01}) is reversed.

\end{lemma}

\begin{theorem}\label{r3} The following inequalities hold for $r,s \in (c,\infty)$,
\begin{enumerate}
\item $\psi(\sqrt{rs})\geq \sqrt{\psi(r)\psi(s)},$\\
equality holds with $r=s$,
\item $\psi(x^k)<\psi(r)^k,\quad k\in(0,1),$
\item $\psi(r)^k<\psi(r^k),\quad k>1$.
\item $\displaystyle\frac{r+s+\psi(r+s)}{r\,\psi(s)+s\,\psi(r)}
\geq \frac{r+s}{r\,s},\quad r,s\in(c,\infty)$
\end{enumerate}
\end{theorem}

\begin{proof}
Let $f(x)=\log(\psi(e^x)),\,x>t=0.379554$, where $t$ is the solution of the equation $e^t=c$.
Differentiating $f$ with respect to $x$ we get by Lemma \ref{lem1} 
\begin{eqnarray*}
f''(x)&=&\frac{e^x \left(\psi \left(e^x\right) \left(\psi'\left(e^x\right)+e^x \psi ''\left(e^x\right)\right)
-e^x \psi '\left(e^x\right)^2\right)}{\psi\left(e^x\right)^2}\\
&<&\frac{e^x \left(\psi \left(e^x\right) \left(\psi'\left(e^x\right)+1-2e^x\psi'(e^x)\right)
-e^x \psi '\left(e^x\right)^2\right)}{\psi\left(e^x\right)^2}\\
&=&\frac{e^x \left(\psi \left(e^x\right) \left(1-(2e^x-1)(1+1/(2e^{2x}))\right)
-e^x \psi '\left(e^x\right)^2\right)}{\psi\left(e^x\right)^2}<0.
\end{eqnarray*}
Hence $f$ is concave, this implies that
$$\frac{\log(\psi(e^x))+\log(\psi(e^y))}{2}\leq \log(\psi(e^{(x+y)/2})).$$
If we let $r=e^x$ and $s=e^y$ we get (1).
The proofs of part (2) and (3) follow from Lemma \ref{rv}(2).
For the proof of part (4), let $$f_1(x)=\frac{\psi(x)/x-1}{x},\quad x>c.$$ Differentiating with respect $x$ we get
\begin{eqnarray*}
f'_1(x)&=&\frac{\psi'(x) x+x-2 \psi(x)}{x^3}\\
&>&\frac{2 x^2+2 x-4 x\log (x)+3}{2 x^4}>0,
\end{eqnarray*}
by Lemma \ref{lem1}(1) and (2). Now the part (4) follows from Lemma \ref{rv}(2).
\end{proof}

\end{document}